\documentclass[12pt,reqno]{amsart}

\usepackage{amssymb,amsthm}
\usepackage{amsfonts,cmtiup,comment,stmaryrd}

\usepackage{mathrsfs,cmtiup,mathabx}

\makeatletter
\def\blfootnote{\xdef\@thefnmark{}\@footnotetext}
\makeatother

\newtheorem{theorem}{Theorem}[section]
\newtheorem{lemma}[theorem]{Lemma}

\newtheorem{corollary}[theorem]{Corollary}

\theoremstyle{definition}

\newtheorem{example}[theorem]{Example}
\newtheorem{remark}[theorem]{Remark}
\newtheorem*{definition*}{Definition}

\newcommand{\N}{\mathbb N}
\newcommand{\F}{\mathbb F}
\newcommand{\Z}{\mathbb Z}

\newcommand{\f}{\varphi}

\newcommand{\g}{\gamma }

\renewcommand{\geq}{\geqslant}
\renewcommand{\leq}{\leqslant}

\newcommand{\ed} {\end{document}}

\let\leq=\leqslant
\let\geq=\geqslant
\numberwithin{equation}{section}

\begin{document}
\title{Profinite groups with an automorphism of prime order\\ whose fixed points have finite Engel sinks}

\author{E. I. Khukhro}
\address{Charlotte Scott Research Centre for Algebra, University of Lincoln, U.K.}
\email{khukhro@yahoo.co.uk}

\author{P. Shumyatsky}

\address{Department of Mathematics, University of Brasilia, DF~70910-900, Brazil}
\email{pavel@unb.br}

\keywords{Profinite groups; Engel condition; locally nilpotent; automorphism}
\subjclass[2010]{Primary 20E18,  20E36; Secondary 20F19, 20F45}

\begin{abstract}
A right Engel sink of an element $g$ of a group $G$ is a set ${\mathscr R}(g)$ such that for every $x\in G$ all sufficiently long commutators $[...[[g,x],x],\dots ,x]$ belong to ${\mathscr R}(g)$.  (Thus, $g$ is a right Engel element precisely when we can choose ${\mathscr R}(g)=\{ 1\}$.) We prove that if a profinite group $G$ admits a coprime automorphism $\f $ of prime order such that every fixed point of $\f$ has a finite right Engel sink,  then $G$ has an open locally nilpotent subgroup.

A left Engel sink of an element $g$ of a group $G$ is a set ${\mathscr E}(g)$ such that for every $x\in G$ all sufficiently long commutators $[...[[x,g],g],\dots ,g]$ belong to ${\mathscr E}(g)$.  (Thus, $g$ is a left Engel element precisely when we can choose ${\mathscr E}(g)=\{ 1\}$.) We prove that if a profinite group $G$ admits a coprime automorphism $\f $ of prime order such that every fixed point of $\f$ has a finite left Engel sink, then $G$ has an open pronilpotent-by-nilpotent subgroup. 
\end{abstract}
\maketitle

\section{Introduction}\label{s-0}

Let $G$ be a profinite group, and $\varphi$ a (continuous) automorphism of $G$ of finite order. We say for short that $\varphi$ is a \textit{coprime automorphism} of $G$ if its order is coprime to the orders of elements of $G$ (understood as Steinitz numbers), in other words, if $G$ is an inverse limit of finite groups of order coprime to the order of $\varphi$. Coprime automorphisms of profinite groups have many properties similar to the properties of coprime automorphisms of finite groups. In particular, if $\varphi$ is a coprime automorphism of $G$, then for any (closed) normal $\varphi$-invariant subgroup $N$ the fixed points of the induced automorphism (which we denote by the same letter) in $G/N$ are images of the fixed points in $G$, that is, $C_{G/N}(\varphi )=C_G(\varphi )N/N$. Therefore, if $\varphi$ is a  coprime automorphism  of prime order $p$ such that $C_G(\varphi )=1$, Thompson's theorem  \cite{tho} implies that $G$ is pronilpotent, and  Higman's theorem~\cite{hig}  implies that $G$ is nilpotent of class bounded in terms of $p$.

In our joint paper with Acciarri \cite{aks} we considered profinite groups admitting a coprime automorphism of prime order all of whose  fixed points are right Engel elements. Recall  that
the $n$-Engel word $[y,{}_{n}x]$ is defined recursively by
$[y,{}_{0}x]=y$ and $[y,{}_{i+1}x]=[[y,{}_{i}x],x]$. An element $g$ of a group $G$ is said to be right Engel if for any $x\in G$ there is an integer $n=n(g,x)$ such that $[g,{}_nx]=1$.  If all elements of a group are right Engel (therefore also left Engel), then the group is called an  Engel group. By a theorem of Wilson and Zelmanov \cite{wi-ze} based on Zelmanov's results \cite{ze92,ze95,ze17} on Engel Lie algebras, an Engel  profinite group is locally nilpotent. Recall that a group is said to be locally nilpotent if every finite subset generates a nilpotent subgroup. The following theorem was proved in \cite{aks}.

\begin{theorem}[{\cite{aks}}]\label{t-aks}
 Suppose that $\varphi$ is a coprime automorphism of prime order of a profinite group $G$. If every element of $C_G(\varphi )$ is a right Engel element of $G$, then $G$ is locally nilpotent.
\end{theorem}

In this paper we consider profinite groups admitting a coprime automorphism of prime order all of whose  fixed points  have finite Engel sinks. Recall that  Engel sinks are used to study generalizations of Engel conditions and are defined as follows.

\begin{definition*} \label{dl}
 A \textit{left Engel sink} of an element $g$ of a group $G$ is a set ${\mathscr E}(g)$ such that for every $x\in G$ all sufficiently long commutators $[x,g,g,\dots ,g]$ belong to ${\mathscr E}(g)$, that is, for every $x\in G$ there is a positive integer $l(x,g)$ such that
 $[x,\,{}_{l}g]\in {\mathscr E}(g)$ for all $l\geq l(x,g).
 $
 \end{definition*}
 \noindent (Thus, $g$ is a left Engel element precisely when we can choose ${\mathscr E}(g)=\{ 1\}$, and $G$ is an Engel group when we can choose ${\mathscr E}(g)=\{ 1\}$ for all $g\in G$.)

 \begin{definition*} \label{dr}
 A \textit{right Engel sink} of an element $g$ of a group $G$ is a set ${\mathscr R}(g)$ such that for every $x\in G$ all sufficiently long commutators $[g,x,x,\dots ,x]$ belong to ${\mathscr R}(g)$, that is, for every $x\in G$ there is a positive integer $r(x,g)$ such that
 $[x,\,{}_{r}g]\in {\mathscr R}(g)$ for all $r\geq r(x,g).
 $
 \end{definition*}
 \noindent (Thus, $g$ is a right Engel element precisely when we can choose ${\mathscr R}(g)=\{ 1\}$, and $G$ is an Engel group when we can choose ${\mathscr R}(g)=\{ 1\}$ for all $g\in G$.)

Our main result concerning right Engel sinks is as follows.

\begin{theorem}\label{t-r}
Let $G$ be a profinite group admitting a coprime automorphism $\f$ of prime order such that all fixed points of $\f$ have finite right Engel sinks. Then $G$ has an open locally nilpotent subgroup.
\end{theorem}

Note that if all elements of a profinite or compact group have finite or even countable left or right Engel sinks, then the group has a finite subgroup with locally nilpotent quotient \cite{khu-shu162,khu-shu191,khu-shu172,khu-shu201}. 
Examples show that such a stronger conclusion does not hold under the hypotheses  of Theorem~\ref{t-r}, which is in a sense best-possible.

One of the important tools in the proof of Theorem~\ref{t-r}  is a strengthened version of Neumann's theorem about $BFC$-groups from the recent paper of Acciarri and Shumyatsky~\cite{acc-shu}. The proof also  makes use of the quantitative version for finite groups that we proved earlier  in~\cite{khu-shu204}. In that paper \cite{khu-shu204} we also proved that if a finite group $G$ has a coprime automorphism $\f$ of prime order such that all fixed points of $\f$ have left Engel sinks of cardinality at most $m$, then $G$ has a metanilpotent subgroup of index bounded in terms of $m$ (examples show that here ``metanilpotent'' cannot be replaced by ``nilpotent''). We prove the following profinite analogue of this result.

\begin{theorem}\label{t-l}
Let $G$ be a profinite group admitting a coprime automorphism $\f$ of prime order $p$. If all fixed points of $\f$ have finite left Engel sinks, then $G$ has an open subgroup that is an extension of a pronilpotent group by a nilpotent group of class $h(p)$, where $h(p)$ is Higman's function depending only on~$p$.
\end{theorem}

 There are examples showing that in the conclusion of Theorem~\ref{t-l} ``pronilpotent-by-nilpotent'' cannot be replaced even by  ``pronilpotent'',   in contrast to the stronger virtual local nilpotency conclusion of Theorem~\ref{t-r} about right Engel sinks. Similarly, if all fixed points of $\f$ are left Engel elements, then the group $G$ is an extension of a pronilpotent group by a nilpotent group of class $h(p)$, where $h(p)$ is Higman's function (Remark~\ref{r-l1}), but $G$ does not have to have an open locally nilpotent subgroup, unlike for the right Engel condition in Theorem~\ref{t-aks}.  Thus, the situation with Engel sinks for fixed points of an automorphism is markedly different from the aforementioned results with conditions on Engel sinks of all elements of a profinite or compact group, where the finiteness (or countability) of right or left Engel sinks resulted in the same conclusion that the group is finite-by-(locally nilpotent).
 
It is worth mentioning that if, under the hypotheses of Theorems~\ref{t-r} (or \ref{t-l}), there is $m\in \N$ such that all right (respectively, left) Engel sinks of fixed points of $\f$ have cardinality at most $m$, then the conclusions can be strengthened, with bounds for the index of a locally nilpotent (respectively, pronilpotent-by-nilpotent) subgroup (Remarks~\ref{r-r} and ~\ref{r-l2}).
 
  We present preliminary material on profinite groups and left and right Engel sinks in \S\,\ref{s-0}. Theorems~\ref{t-r} and \ref{t-l} about right and left Engel sinks are proved in \S\,\ref{s-r} and \S\,\ref{s-l}, respectively. In \S\,\ref{s-e} we present examples showing that in some respects Theorems \ref{t-r} and \ref{t-l} cannot be improved.

\section{Preliminaries}\label{s-p}
In this section we recall some definitions and general properties related to profinite groups and Engel sinks.

Our notation and terminology for profinite groups is standard; see, for example,  \cite{rib-zal} and  \cite{wil}.  A subgroup (topologically) generated by a subset $S$ is denoted by $\langle S\rangle$. By a subgroup we always mean a closed subgroup, unless explicitly stated otherwise. Recall that centralizers are closed subgroups, while commutator subgroups $[B,A]=\langle [b,a]\mid b\in B,\;a\in A\rangle$ are the closures of the corresponding abstract commutator subgroups.

For a group $A$ acting by automorphisms on a group $B$ we use the usual notation for commutators $[b,a]=b^{-1}b^a$ and commutator subgroups $[B,A]=\langle [b,a]\mid b\in B,\;a\in A\rangle$, as well as for centralizers $C_B(A)=\{b\in B\mid b^a=b \text{ for all }a\in A\}$
and $C_A(B)=\{a\in A\mid b^a=b\text{ for all }b\in B\}$. A section $A/B$ of a group $G$ is a quotient of a subgroup $A\leq G$ by a normal subgroup $B$ of $A$. The centralizer of a section is $C_G(A/B)=\{g\in G\mid [A,g]\leq B\}$. The definition and some properties of coprime automorphisms of profinite groups were already mentioned at the beginning of the Introduction in \S\,\ref{s-0}.

Recall that a pro-$p$ group  is an inverse limit of finite $p$-groups, a pronilpotent group is an inverse limit of finite nilpotent groups, a prosoluble group is an inverse limit of finite soluble  groups. We denote by $\pi (k)$ the set of prime divisors of $k$, where $k$ may be a positive integer or a Steinitz number, and by $\pi (G)$ the set of prime divisors of the orders of elements of a (profinite) group $G$. Let $\sigma$ be a set of primes. An element $g$ of a group is  a $\sigma$-element if $\pi(|g|)\subseteq \sigma$, and a group $G$ is a $\sigma$-group if all of its elements are $\sigma$-elements. We denote by $\sigma'$ the complement of $\sigma$ in the set of all primes. When $\sigma=\{p\}$,  we write $p$-element, $p'$-element, etc.
Profinite groups have Sylow $p$-subgroups and satisfy analogues of the Sylow theorems.  Prosoluble groups satisfy analogues of the theorems  on Hall $\pi$-subgroups.  We refer the reader to the corresponding chapters in \cite[Ch.~2]{rib-zal} and \cite[Ch.~2]{wil}.

We denote by  $\gamma _{\infty}(G)=\bigcap _i\gamma _i(G)$ the intersection of the lower central series of a group~$G$. A profinite group $G$ is pronilpotent if and only if $\gamma _{\infty}(G)=1$, which is also equivalent to $G$ being the Cartesian product of its Sylow subgroups. Every profinite group $G$ has a maximal normal pronilpotent subgroup denoted by $F(G)$. This subgroup has the following characterization, similar to that of the Fitting subgroup of a finite group.

  \begin{lemma}\label{l-f}
  The maximal normal pronilpotent subgroup $F(G)$ of a profinite group $G$ is equal to the intersection of the centralizers of all chief factors of all finite quotients of $G$ by open normal subgroups.
    \end{lemma}

  \begin{proof}
    The intersection in question is clearly a closed normal subgroup. In any finite quotient of $G$, the image of this intersection is nilpotent by the well-known characterization of the Fitting subgroup of a finite group \cite[5.2.9]{rob}. Hence this intersection is contained in $F(G)$. Conversely, any element of $F(G)$ clearly belongs to the Fitting subgroup of any finite quotient of $G$ and therefore centralizes every chief factor of it.
  \end{proof}

We can define a profinite analogue of the Fitting series by setting $F_1(G)=F(G)$, and then by induction $F_{k+1}(G)$ being the inverse image of $F(G/F_k(G))$. It is natural to say that a profinite group has pronilpotent length $l$ if $F_l(G)=G$ and $l$ is minimal with this property.
We record a useful elementary lemma about the pronilpotent series.

  \begin{lemma}\label{l-f2}
  {\rm (a)}
  If $H$ is a  subgroup of a profinite group $G$ such that $F(G)\leq H\leq F_2(G)$, then $F(H)=F(G)$.

  {\rm (b)} If $g$ is a $p$-element  in $F_2(G)\setminus F(G)$, then $g$ induces by conjugation a non-trivial automorphism of the Hall $p'$-subgroup of $F(G)$.
  \end{lemma}

\begin{proof}
(a)  Clearly, $F(G)\leq F(H)$. We now prove the reverse inclusion. Any chief factor $A/B$ of a finite quotient of $G$ is a section of $G/F_2(G)$, or of $F_2(G)/F(G)$, or of $F(G)$. An element of $F(H)$ centralizes $A/B$ in the first case because $H\leq F_2(G)$, in the second case because $F_2/F(G)$ is pronilpotent, and in the third case because $F(G)\leq H$. Hence $F(H)\leq F(G)$ by Lemma~\ref{l-f}.

(b) By Lemma~\ref{l-f} the element $g$ must act nontrivially on some chief factor of a finite quotient of $G$ by an  open normal subgroups. Since $g\in F_2(G)$, such a chief factor must be a section of $F(G)$, and since $ g$ is contained in a Sylow $p$-subgroup containing the  Sylow $p$-subgroup of $F(G)$, such a chief factor must be a section of the Hall $p'$-subgroup of $F(G)$.
  \end{proof}

  If  $P$ is a pro-$p$ group, the  Frattini subgroup of $P$ is $\Phi (P)=[P,P]P^p$. If $\alpha$ is a coprime automorphism of $P$, then $\alpha $ acts nontrivially on $P/\Phi(P)$. The Frattini subgroup of a pronilpotent group is the Cartesian product of the Frattini subgroups of its Sylow $p$-subgroups. It follows from Lemmas~\ref{l-f} and \ref{l-f2} that
\begin{equation}\label{e-phi}
  F(G/\Phi (F(G)))=F(G)/\Phi (F(G)).
\end{equation}

\begin{lemma}\label{l-gam} Let $G$ be a profinite group such that $\g_{\infty}(G)$ is finite. Then $C_G(\g_{\infty}(G))$ is an open pronilpotent subgroup.
\end{lemma}

\begin{proof}
The subgroup $C_G(\g_{\infty}(G))$ is closed and has finite index, since $G/C_G(\g_{\infty}(G))$ faithfully acts by automorphisms on $\g_{\infty}(G)$; hence $C_G(\g_{\infty}(G))$ is an open normal subgroup. Any  chief factor $A/B$ of a finite quotient of $G$ by an open normal subgroup is either a section of $G/\g_{\infty}(G)$, which is pronilpotent, or of $\g_{\infty}(G)$. Hence any  element of $C_G(\g_{\infty}(G))$ centralizes $A/B$ and the result follows by Lemma~\ref{l-f}.
\end{proof}

We recall the well-known consequence of the Baire Category Theorem (see \cite[Theorem~34]{kel}).

\begin{theorem}\label{bct}
If a profinite group is a countable union of closed subsets, then one of these subsets has non-empty interior.
\end{theorem}

We now recall some general properties of Engel sinks.
Clearly, the intersection of two left Engel sinks of a given element $g$ of a group $G$ is again a left Engel sink of $g$, with the corresponding function $l(x,g)$ being the maximum of the two functions. Therefore, if $g$ has a \textit{finite} left Engel sink, then $g$ has a unique smallest left  Engel sink, which  has the following  characterization.

\begin{lemma}[{\cite[Lemma~2.1]{khu-shu162}}]\label{l-min} If an element $g$ of a group $G$ has a finite left Engel sink, then $g$ has a smallest left Engel sink $\mathscr E (g)$
and for every $s\in \mathscr E (g)$ there is an integer $k\geq 1$ such that  $s=[s,\,{}_kg]$.
\end{lemma}

 The intersection of two right Engel sinks of a given element $g$ of a group $G$ is again a right Engel sink of $g$, with the corresponding function $r(x,g)$ being the maximum of the two functions. Therefore, if $g$ has a \textit{finite} right Engel sink, then $g$ has a unique smallest right Engel sink, which is henceforth denoted by ${\mathscr R}(g)$. It has  the following  characterization.

\begin{lemma}[{\cite[Lemma~2.2]{khu-shu172}}]\label{l-min-r}
If an element $g$ of a group $G$ has a finite right Engel sink, then $g$ has a smallest right Engel sink $\mathscr R(g)$ and for every $z\in \mathscr R(g)$ there are integers $n\geq 1$ and $m\geq 1$ and an element $x\in G$ such that $z=[g,{}_nx]=[g,{}_{n+m}x]$.
\end{lemma}

\noindent (Here, the elements $x$ and numbers $m,n$ can be different for different $z$.)

Furthermore, for metabelian groups we have the following.

\begin{lemma}[{\cite[Lemma~2.5]{khu-shu172}}]\label{l-metab}
 If $G$ is a metabelian group, then a right Engel sink of the inverse $g^{-1}$ of an element $g\in G$ is a left Engel sink of $g$.
\end{lemma}

\begin{remark} If $\f$ is an automorphism of finite order $p$ of a profinite group $G$ and $H$ is an open normal subgroup of $G$, then $\bigcap_{i=0}^{p-1}H^{\f ^i}$ is a $\f$-invariant open normal subgroup. Thus, $\f$\text{-invariant} open normal subgroups of $G$ form a base of neighbourhoods of $1$ in the profinite topology. We  freely use this property throughout the paper without special references.
\end{remark}

\begin{remark} If every element of a subgroup $C$ has a finite right Engel sink in a group $G$, then this condition is inherited by the image of $C\cap A$ in every section $A/B$, and we shall use this property without special references. The same applies to a subgroup in which every element has a finite left Engel sink. \end{remark}

 Throughout the paper, we write, say, ``$(a,b,\dots )$-bounded'' to abbreviate ``bounded above in terms of $a, b,\dots $ only''.

\section{Right Engel sinks}\label{s-r}

In this section we prove Theorem~\ref{t-r} concerning right Engel sinks of fixed points of an automorphism.

\begin{lemma}\label{l-pron}
  If $G$ is a pronilpotent group and an element $g\in G$ has a finite right Engel sink, then in fact $ \mathscr R(g)=\{1\}$, that is, $g$ is a right Engel element.
\end{lemma}

\begin{proof}
  Since $ \mathscr R(g)$ is finite, there is an open normal subgroup $N$ such that $ \mathscr R(g)\cap N=\{1\}$. If $z\in \mathscr R(g)$, then by Lemma~\ref{l-min-r}  there are integers $n\geq 1$ and $m\geq 1$ and $x\in G$ such that $z=[g,{}_nx]=[g,{}_{n+m}x]$. Therefore the image of $z$ in $G/N$ must be trivial, since $G/N$ is nilpotent. Hence $z\in N\cap \mathscr R(g)=\{1\}$.
\end{proof}

Combining Lemma~\ref{l-pron} with Theorem~\ref{t-aks} we obtain the following.

\begin{corollary}\label{c-pron}
     If  a pronilpotent group $G$ admits a coprime automorphism of prime order such that every fixed point has a finite right Engel sink, then $G$ is locally nilpotent.
\end{corollary}

The following quantitative version of Theorem~\ref{t-r} for finite groups was proved in \cite{khu-shu204}.

\begin{theorem}[{\cite[Theorem~1.4]{khu-shu204}}]\label{t-bmjr}
Let $G$ be a finite group admitting an automorphism $\f$ of prime order coprime to $|G|$. Let $m$ be a positive integer such that every element $g\in C_G(\f)$ has a right Engel sink ${\mathscr R}(g)$ of cardinality at most $m$. Then $G$ has a nilpotent normal subgroup of $m$-bounded index.
\end{theorem}

In the proof of Theorem~\ref{t-r} we will combine this result  with Corollary~\ref{c-pron} and a reduction to the case of uniformly bounded sizes of right Engel sinks of fixed points.

\begin{proof}[Proof of Theorem~\ref{t-r}]
Recall that $G$ is a profinite group admitting a coprime automorphism $\f$ of prime order such all fixed points of $\f$ have finite right Engel sinks; we need to produce an open locally nilpotent subgroup.
By Corollary~\ref{c-pron} any  $\f$-invariant pronilpotent subgroup of $G$ is locally nilpotent and therefore it is sufficient to produce an open pronilpotent subgroup.

Let $g\in C_G(\f)$  and let $N_g$ be an  open normal subgroup such that $N_g\cap {\mathscr R}(g)=\{1\}$. Then $g$ is a right Engel element of the subgroup $N_g\langle g\rangle $. By Baer's theorem \cite[12.3.7]{rob}, in every finite quotient of $N_g\langle g\rangle$ the image of $g$ belongs to the hypercentre. Therefore the subgroup $[N_g, g]$ is pronilpotent.

Let $\widetilde N_g$ be the normal closure of $[N_g, g]$ in $G$. Since $[N_g, g]$ is normal in the subgroup $N_g$, which has finite index, $[N_g, g]$ has only finitely many conjugates. Hence $\widetilde N_g$ is a product of finitely many normal subgroups of $N_g$, each of which is pronilpotent, and therefore $\widetilde N_g$ is pronilpotent. Therefore all the subgroups $\widetilde N_g$ are contained in the largest normal pronilpotent subgroup $F(G)$.

The image $\bar g$ of every element $g\in C_G(\f )$ in $\widebar G=G/F(G)$ has finite conjugacy class $\bar g^G$, since $[g,N_g]\leq F(G)$ and $N_g$ has finite index in $G$. We now use a strengthened version of Neumann's theorem about $BFC$-groups and a lemma about finite conjugacy classes in profinite groups from the recent paper of Acciarri and Shumyatsky~\cite{acc-shu}. Namely, by \cite[Lemma~4.2]{acc-shu} there is an integer $n$ such that $|\bar g^G|\leq n$ for every  $\bar g\in  C_{\widebar{G}}(\f )$. Let $H = \langle C_{\widebar{G}}(\f )^G\rangle $ be the abstract normal closure of $C_{\widebar{G}}(\f )$ in $\widebar G$. Then by  \cite[Theorem~1.1]{acc-shu} the derived subgroup $H'$ is finite (of $n$-bounded order). In particular, $H'$ is a closed subgroup of $\widebar G$. Let $\widetilde H$  be the topological closure of $H$ in $\widebar G$. Since $H/H'$ is abelian, $\widetilde H/H'$ is also abelian. (We had to consider the abstract normal closure first, since \cite[Theorem~1.1]{acc-shu} is stated for abstract groups; but it is clear that it also works for profinite groups as shown above.)

Note that  $C_{\widebar{G}}(\f )\leq \widetilde H$ and therefore $\widebar G/\widetilde H$ is nilpotent by the theorems of Thompson \cite{tho} and Higman \cite{hig}. Let $N$ be a   $\f$-invariant open normal subgroup of $G$ containing $F(G)$ such that $\widebar N\cap H'=1$. Then $N/F(N)$ is abelian-by-nilpotent. Replacing $G$ with $N$ we can assume from the outset that $G/F(G)$ is soluble and proceed by induction on the derived length of it.

The main case is when $G/F(G)$ is abelian. Indeed,  in the general case,  by induction hypothesis,   $G'F(G)$ has a $\f$-invariant open pronilpotent  subgroup $M$. Since $G'M/M$ is finite, 
there is a $\f$-invariant open normal  subgroup $N$   such that $N\cap G' =M$. Note that $F(N)\geq M$. Then $N/F(N)$ is abelian, so we may assume that $G/F(G)$ is abelian from the outset. We need to show that $G/F(G)$ is finite.

We write $F=F(G)$ to lighten the notation.  Since $F(G/\Phi (F))=F/\Phi (F)$ by \eqref{e-phi},  we can assume that $\Phi (F)=1$. In particular, then $G$ is metabelian.

\begin{lemma}\label{l-cfin}
   $C_G(\f)F/F$ is finite.
\end{lemma}

\begin{proof}
  By Lemma~\ref{l-metab} we have ${\mathscr E}(g)\subseteq {\mathscr R}(g^{-1})$ in a metabelian group. Hence all elements of $C_G(\f)$ have finite left Engel sinks. Every subset $E_k=\{x\in C_G(\f )\mid |{\mathscr E}(x)|\leq k\}$ is closed in the induced topology of  $C_G(\f )$. Indeed,  this is equivalent to the complement of $C_G(\f )\setminus E_k$ being an open subset of $C_G(\f )$. For every element $g\in C_G(\f )\setminus E_k$ we have $|{\mathscr E}(g)|\geq k+1$, so there are distinct elements $z_1,z_2,\dots ,z_{k+1}\in {\mathscr E}(g)$. By Lemma~\ref{l-min}   we can write for every $i=1,\dots ,k+1$
\begin{equation}\label{e-open}
z_i=[z_i,g,\dots ,g], \quad \text{where } g \text{ is repeated } k_i\geq 1 \text{ times}.
 \end{equation}
 Let $N$ be an open normal subgroup of $G$ such that the images of $z_1,z_2,\dots ,z_{k+1}$ are distinct elements in $G/N$. Then equations \eqref{e-open} show that for any $u\in N\cap C_G(\f )$ the Engel sink $ {\mathscr E}(gu)$ contains an element in each of the $k+1$ cosets $z_iN$. This means that the whole coset $g(N\cap C_G(\f ))$ is contained in $C_G(\f )\setminus E_k$. Thus every element of $C_G(\f )\setminus E_k$ has a neighbourhood contained in $C_G(\f )\setminus E_k$, which is therefore an open subset of $C_G(\f )$.

  Since $C_G(\f )=\bigcup_k E_k$, by the Baire Category Theorem~\ref{bct} there is $m\in \N$, an open (in the induced topology) normal subgroup $C_1$, and a coset $c_0C_1$ such that $|{\mathscr E}(c_0x)|\leq m$ for any $x\in C_1$. We now obtain that $|{\mathscr E}(x)|$ is $m$-bounded for any $x\in C_1$. 
  Indeed,  by \cite[Lemma~2.5]{khu-shu172}, in a metabelian group, if ${\mathscr E}(g)$ is finite, then ${\mathscr E}(g)$ is a normal subgroup. In  the quotient $\widebar M=M/\big({\mathscr E}(c_0) {\mathscr E}(c_0x)\big)$,  both $\widebar M'\langle \bar c_0\rangle$ and $\widebar M' \langle \bar c_0 \bar x\rangle$ are normal locally nilpotent subgroups. Hence their product, which contains $\bar x$, is also a locally  nilpotent subgroup by the Hirsch--Plotkin theorem \cite[12.1.2]{rob}. As a result, ${\mathscr E}(x)\leq {\mathscr E}(c_0) {\mathscr E}(c_0x)$ and therefore
\begin{equation}\label{e-lbound}
  |{\mathscr E}(x)|\leq |{\mathscr E}(c_0)| \cdot |{\mathscr E}(c_0x)|\leq m^2 \qquad \text{ for any }x\in C_1.
\end{equation}

To prove that $C_G(\f)F/F$ is finite, it remains to show that $C_1F/F$ is finite. For that we use the following lemma.

\begin{lemma}[{\cite[Lemma~2.4]{khu-shu204}}]\label{l-local}
  Suppose that $V$ is an abelian  finite group, and $U$ a group of coprime
automorphisms of $V$.  If $|[V,u]|\leq n$ for every $u\in U$, then $|[V,U]|$ is $n$-bounded, and therefore $|U|$ is also $n$-bounded.
\end{lemma}

  Let $q$ be any prime in  $\pi (C_1 F/ F)$, and let $Q$ be a Sylow $q$-subgroup of $ C_1 F$. Let $f(n)$ be the function furnished by Lemma~\ref{l-local} as a bound in terms of $n$ for $|U|$.  We claim that $|QF/F|\leq f(m^2)$; this will imply that $C_1F/F$ is finite.  Note that every element of $Q\setminus F$ acts non-trivially on the Hall $q'$-subgroup $F_{q'}$ of $F$. For every $u\in Q$, the left Engel sink  ${\mathscr E}(u)$ is a normal subgroup of order $\leq m^2$ by \eqref{e-lbound}. Since  $F_{q'}\langle u\rangle/{\mathscr E}(u)$ is pronilpotent and $u$ induces a coprime automorphism on $F_{q'}$, we have  $F_{q'}={\mathscr E}(u)  C_{F_{q'}}(u)$.
Hence $[F_{q'},u]=[{\mathscr E}(u)  C_{F_{q'}}(u),u]=[{\mathscr E}(u), u]={\mathscr E}(u)$, where the last equality holds by the minimality of ${\mathscr E}(u)$. Therefore in every  finite quotient $\widebar G$ of $G$ by a $\f$-invariant open normal subgroup we have $|\widebar Q/C_{\widebar Q}(\widebar F_{q'})|\leq f(m^2)$ by Lemma~\ref{l-local}. Hence $|QF/F|=|Q/C_{Q}(F_{q'})|\leq f(m^2)$, as claimed.

Since $C_1$ has finite index in $C_G(\f )$ and $C_1F/F$ is finite, we conclude that $C_G(\f )F/F$ is finite.
\end{proof}

Let $N$ be a $\f$-invariant open normal subgroup of $G$ containing $F$ such that $N\cap C_G(\f )\leq F$. Then Since $F(N)= F$, we have $C_N(\f)\leq F(N)$. Replacing $G$ with $N$ we can assume from the outset that $C_G(\f)\leq F$.

Every subset $R_k=\{x\in C_G(\f )\mid |{\mathscr R}(x)|\leq k\}$ is closed in the induced topology of  $C_G(\f )$. Indeed, this is equivalent to the complement of $R_k$ being an open subset of $C_G(\f )$. For every element $g\in C_G(\f )\setminus R_k$ we have $|{\mathscr R}(g)|\geq k+1$ and there are distinct elements  $z_1,z_2,\dots ,z_{k+1}\in {\mathscr R}(g)$. Using Lemma~\ref{l-min-r}
we can write for every $i=1,\dots ,k+1$
\begin{equation}\label{e-closed}
z_i=[g,{}_{n_i}x_i]=[g,{}_{n_i+m_i}x] \quad \text{for some } x_i\in G,\; n_i\geq 1,\;m_i\geq 1.
 \end{equation}
 Let $N$ be an open normal subgroup of $G$ such that the images of $z_1,z_2,\dots ,z_{k+1}$ are distinct elements in $G/N$. Then equations \eqref{e-closed} show that for any $u\in N\cap C_G(\f )$ the right Engel sink $ {\mathscr R}(gu)$ contains an element in each of the $k+1$ cosets $z_iN$. Thus, the coset $g(N\cap C_G(\f ))$ is contained in $C_G(\f )\setminus R_k$. This means that every element of $C_G(\f )\setminus R_k$ has a neighbourhood contained in $C_G(\f )\setminus R_k$, which is therefore an open subset of $C_G(\f )$.

Since $C_G(\f )=\bigcup _kR_k$, by the Baire Category Theorem~\ref{bct} there is $m\in \N$, an open (in the induced topology) normal subgroup $C_1$ of $C_G(\f )$, and a coset $c_0C_1$ such that
\begin{equation}\label{e-rbound}
  |{\mathscr R}(c_0x)|\leq m\quad \text{ for any } x\in C_1.
\end{equation}

By the standard commutator formula  $[xy,z]=[x,z]^y[y,z]$, using the fact that 
$F$ is abelian we have
$$
[ab,\underbrace{x,\dots ,x}_n]=[a,\underbrace{x,\dots ,x}_n]\cdot [b,\underbrace{x,\dots ,x}_n]
$$
for any $a,b\in F$, any  $x\in G$, and any $n\in \N$. Therefore  we obtain that ${\mathscr R}(ab)  \subseteq {\mathscr R}(a) {\mathscr R}(b)$ for any $a,b\in C_G(\f)\leq F$. The same commutator formula shows that  ${\mathscr R}(a^{-1})= \{z^{-1}\mid z\in {\mathscr R}(a)\}$, so that $|{\mathscr R}(a^{-1})|= | {\mathscr R}(a)|$.  From \eqref{e-rbound} we now obtain
\begin{equation}\label{e-rbound2}
  |{\mathscr R}(x)|\leq |{\mathscr R}(c_0^{-1})| \cdot |{\mathscr R}(c_0x)|\leq m^2 \qquad \text{ for any }x\in C_1.
\end{equation}

Let $N$ be a $\f$-invariant open normal subgroup of $G$ such that $N\cap C_G(\f)=C_1$. Then $ |{\mathscr R}(x)|\leq m^2$  for any $x\in C_N(\f)$.
By Theorem~\ref{t-bmjr}, every finite quotient of $N$ by a $\f$-invariant open normal subgroup has a nilpotent 
subgroup of index at most $f(m^2)$ for some function $f(n)$ depending on $n$ alone. Therefore $N$ has a $\f$-invariant  open pronilpotent 
subgroup $M$ of index at most $f(m^2)$, 
which  is  locally nilpotent
by Corollary~\ref{c-pron}. Clearly, $M$ is a sought-for open locally nilpotent subgroup of $G$.
\end{proof}

 \begin{remark}\label{r-r}  If, under the hypotheses of Theorem~\ref{t-r} there is a positive integer $n$ such that all fixed points of $\f$ have finite right Engel sinks of cardinality at most $n$, then the group $G$ has
  a locally nilpotent subgroup of finite  $n$-bounded index. This immediately follows from 
Theorem~\ref{t-bmjr} applied to finite quotients of $G$ by a $\f$-invariant open normal subgroup:  every such quotient has a nilpotent subgroup of index at most $f(n)$ for a function $f(n)$ depending only on $n$.  Therefore $G$ has a $\f$-invariant  open pronilpotent 
subgroup $M$ of index at most $f(n)$, 
which  is  locally nilpotent
by Corollary~\ref{c-pron}. 
\end{remark}

\section{Left Engel sinks}\label{s-l}

In this section we prove Theorem~\ref{t-l} concerning left Engel sinks of fixed points. We begin with the following lemma.

  \begin{lemma}\label{l-f2fin}

  Suppose that  $G=F_2(G)$ is a profinite group of pronilpotent length $2$ admitting a coprime automorphism $\f$ of prime order such that all elements of $C_G(\f)$ have finite left Engel sinks.

    {\rm (a)} Then $C_{G/F(G)}(\f)$ is finite.

     {\rm (b)}  If there is $m\in \N$ such that $|{\mathscr E}(c)|\leq m$ for all $c\in C_G(\f)$, then $|C_{G/F(G)}(\f)|$ is $m$-bounded.

      {\rm (c)}  If all elements of $C_G(\f)$ are left Engel elements, then $C_{G}(\f)\leq F(G)$.
  \end{lemma}

  \begin{proof}
  We write $F=F(G)$ to lighten the notation.  Let $\Phi (F)$ be the Frattini subgroup of~$F$.  Since $F(G/\Phi (F))=F/\Phi (F)$ by \eqref{e-phi}, we can assume that $F$ is abelian in all parts of the lemma.
  
(a)     Since the group $FC_G(\f)/F$ is pronilpotent and all its elements have finite left Engel sinks, this group is locally nilpotent by \cite[Lemma~4.2]{khu-shu162}. We further claim that all elements of $FC_G(\f)$ have finite left Engel sinks in $FC_G(\f)$. Indeed, let $g=uc$ and $h=vd$, where  $u,v\in F$ and $c,d\in C_{G}(\f)$. For some $k$ the commutator $[h,\,{}_kg]$ belongs to $F$, since $\langle u,v\rangle F/F$ is nilpotent. Then $[h,\,{}_{k+n}g]=[[h,\,{}_kg],\,{}_nc]$ for any $n$, since $F$ is abelian. As  a result,  ${\mathscr E}(g)$ is contained in ${\mathscr E}(c)$, which is finite by hypothesis.

    Applying \cite[Theorem~1.2]{khu-shu162} we obtain that $\g_{\infty}(FC_G(\f))$ is finite. By Lemma~\ref{l-gam},  $C_G(\g_{\infty}(FC_G(\f)))$ is a closed normal pronilpotent subgroup, which has finite index in $FC_G(\f)$. It follows that $F(FC_G(\f))$ has finite index in $FC_G(\f)$, and the result follows, since $F(FC_G(\f))=F$ by Lemma~\ref{l-f2}.

     (b) If there is $m\in \N$ such that $|{\mathscr E}(c)|\leq m$ for all $c\in C_G(\f)$, then the above argument shows that $|{\mathscr E}(g)|\leq m$ for all $g\in FC_{G}(\f)$. By \cite[Theorem~3.1]{khu-shu162} then $|\g_{\infty}(\widebar FC_{\widebar G}(\f))|$ is $m$-bounded for every finite quotient $\widebar G$ of $G$ by a $\f$-invariant open normal subgroup. Hence $|\g_{\infty}( FC_{ G}(\f))|$ is also $m$-bounded, and so is the index of $F$ in $ FC_{ G}(\f)$ by the above argument.

     (c)  Since all elements of $C_G(\f)$ are left Engel elements, the above argument shows that  $C_{G}(\f)F$ is an Engel group, and therefore pronilpotent. Since  $F=F(FC_G(\f))$ by Lemma~\ref{l-f2}, we obtain $F =FC_G(\f)$.
  \end{proof}

  \begin{proof}[Proof of Theorem~\ref{t-l}]
 Recall that $G$ is a profinite group admitting a coprime automorphism $\f$ of prime order such that all elements of $C_G(\f)$ have finite left Engel sinks. We need to show that $G$ has an open pronilpotent-by-nilpotent subgroup.   First we perform reduction to the case of pronilpotent length $3$, that is, $G=F_3(G)$. Since all elements of $C_G(\f)$ have finite left Engel sinks, by \cite[Theorem~1.2]{khu-shu162} there is a finite normal subgroup $C_0$ of $ C_G(\f)$ such that $C_G(\f)/C_0$ is locally nilpotent. Let $N$ be a $\f$-invariant open normal subgroup of $G$ such that $N\cap C_0=1$. Then $C_N(\f)$ is locally nilpotent. Then the centralizer $C_{\widebar N}(\f)$ is nilpotent for every finite quotient $\widebar N$ of $N$ by a $\f$-invariant open normal subgroup.   Wang and Chen \cite{wan-che} used the classification of finite simple groups to prove that a finite group admitting a coprime automorphism of prime with nilpotent fixed-point subgroup is soluble. Furthermore, by a theorem of Turull \cite{tur} (the best-possible  improvement of the earlier result of Thompson~\cite{tho64}), the Fitting height of $\widebar N$ is at most $3$. Thus, every finite quotient of $N$ by an open normal subgroup has Fitting height at most $3$; hence $N$ has pronilpotent length at most $3$. Replacing $G$ with $N$,  we can assume from the outset that  $G=F_3(G)$.

  By Lemma~\ref{l-f2fin},  both $C_{G/F_2(G)}(\f)$ and $C_{F_2(G)/F(G)}(\f)$ are finite. Hence, $C_{G/F(G)}(\f)$ is finite. Let $H$ be a $\f$-invariant open normal subgroup containing $F(G)$ such that $H\cap C_G(\f)\leq F(G)$. Then $C_H(\f)\leq F(G)=F(H)$. By the theorems of Thompson~\cite{tho} and Higman~\cite{hig} the quotient $H/F(H)$ is nilpotent of class at most $h(p)$, where $p=|\f|$.

  Thus, $G$ has an open subgroup that is an extension of a pronilpotent group by a nilpotent group of class at most $h(p)$.
   \end{proof}

   \begin{remark}\label{r-l1}  If, under the hypotheses of Theorem~\ref{t-l} 
    all fixed points of $\f$ are left Engel elements, then the group $G$ is an extension of a pronilpotent group by a nilpotent group of class $h(p)$, where $h(p)$ is Higman's function. Indeed,  then $C_G(\f)\leq F(G)$ by Lemma~\ref{l-f2fin}(c). Then $G/F(G)$ is nilpotent of class at most $h(p)$, where $p=|\f|$, by the theorems of Thompson~\cite{tho} and Higman~\cite{hig}.
  \end{remark}

 \begin{remark}\label{r-l2}  If, under the hypotheses of Theorem~\ref{t-l} there is a positive integer $m$ such that all fixed points of $\f$ have finite left Engel sinks of cardinality at most $m$, then the group $G$ has
 \begin{itemize}
                                 \item[{\rm (a)}]  a meta-pronilpotent subgroup of finite  $m$-bounded index, and
                                 \item[{\rm (b)}] a subgroup of finite  $(p,m)$-bounded index that is an extension of a   pronilpotent group by a nilpotent group of class $g(p)$, where $p=|\f|$ and $g(p)$ is a function depending only on $p$.
                                       \end{itemize}
 Indeed, then by \cite[Theorem~1.3]{khu-shu204} for every finite quotient $\widebar G$ of $G$  by a $\f$-invariant open normal subgroup, the index of $F_2(\widebar G)$ in $\widebar G$ is $m$-bounded. Hence the index of $F_2(G)$ in $G$ is also $m$-bounded. Furthermore, by Lemma~\ref{l-f2fin}(b) the order $|C_{\bar G/F(\bar G)}(\f)|$ is $m$-bounded. By Khukhro's theorem \cite[Theorem~2]{khu90}, then $\bar G/F(\bar G)$ has a subgroup of $(p,m)$-bounded index that is nilpotent of $p$-bounded class $g(p)$. This implies that $\widebar G$ has an open subgroup of $(p,m)$-bounded index that is  nilpotent of class $g(p)$.
\end{remark}

\section{Examples}\label{s-e}

Here we present examples showing that in some respects Theorems \ref{t-r} and \ref{t-l} cannot be improved.

\begin{example}\label{ex-1}
Let $V$ be an elementary abelian group of order $7^2$, and $D_6=\langle a\rangle\rtimes \langle b\rangle $ a group of automorphisms of $V$ such that $a^3=1$, $b^2=1$, $C_V(a)=1$, and $|C_V(b)|=7$. Let $F=\prod_{i=1}^{\infty}V_i$ be the Cartesian product of isomorphic copies $V_i$ of $V$ as $\F_7D_6$-modules. Then $D_6$ can be regarded as a group of automorphisms of $F$. Let $G=F\langle a\rangle$ and let $\f$ be the automorphism of $G$ of order $2$ induced by $b$. Then $C_G(\f )=\prod_{i=1}^{\infty}C_{V_i}(\f )$. Using the fact that all the $V_i$ are isomorphic $\F_7\langle a \rangle$-modules, one can show that for any $c\in C_G(\f )$ the right Engel sink $\mathscr R(c)$ is finite and, moreover, the sizes of these sinks are uniformly bounded. At the same time, $\g_{\infty}(G)=F$ is infinite. This example shows that under the hypotheses of  Theorem~\ref{t-r} one cannot obtain a finite subgroup with a pronilpotent quotient.
\end{example}

\begin{example}\label{ex-2}
For the same $V$ and $D_6=\langle a\rangle\rtimes \langle b\rangle $ as in Example~\ref{ex-1}, let $F=\prod_{i=1}^{n}V_i$ be a finite direct product of $n$ copies $V_i$ of $V$ as  $\F_7D_6$-modules.  Let $G=F\langle a\rangle$ and let $\f$ be the automorphism of $G$ of order $2$ induced by $b$. Then $C_G(\f )=\prod_{i=1}^{n}C_{V_i}(\f )$. There is a constant $m$ independent of $n$ such that  $|\mathscr R(c)|\leq m$ for any $c\in C_G(\f )$.
In these examples, $\g_{\infty}(G)=F$, so $|\g_{\infty}(G)|$ cannot be bounded in terms of $m$ (and $|\f|$). This shows that the conclusion of Theorem~\ref{t-bmjr} (\cite[Theorem~1.4]{khu-shu204}) also cannot be improved in this respect.
\end{example}

\begin{example}\label{ex-3}
For the same $V$ and $D_6=\langle a\rangle\rtimes \langle b\rangle $ as in Example~\ref{ex-1}, let $H=\prod_{i=1}^{\infty}V_i\rtimes (\langle a_i\rangle \rtimes \langle b_i\rangle)$ be the Cartesian product of isomorphic copies $V_i\rtimes (\langle a_i\rangle \rtimes \langle b_i\rangle)$ of the semidirect product $V\rtimes (\langle a\rangle \rtimes \langle b\rangle)$ with $V_i,a_i,b_i$ naturally corresponding to $V,a,b$. Let $G=\prod_{i=1}^{\infty}V_i\rtimes \langle a_i\rangle $ and let $\f$ be the automorphism of $G$ of order $2$ induced by the `diagonal' $\prod_{i=1}^{\infty} b_i$. Then $F(G)=\prod_{i=1}^{\infty}V_i$ and $C_G(\f )=\prod_{i=1}^{\infty}C_{V_i}(\f )$. Since $F$ is abelian, all left Engel sinks of fixed points of $\f$ are trivial.   At the same time, $G/F(G)$ is infinite. This example shows that under the hypotheses of  Theorem~\ref{t-l} one cannot obtain an open pronilpotent subgroup (and the more so, a finite normal subgroup with a pronilpotent quotient).
\end{example}

\begin{example}\label{ex-4}
Similarly to Example~\ref{ex-2}, using  finite direct products $H=\prod_{i=1}^{n}V_i\rtimes (\langle a_i\rangle \rtimes \langle b_i\rangle)$ instead of the Cartesian product,  we obtain examples of finite groups $G$ with a coprime automorphism $\f$ of order $2$ such that all elements of $C_G(\f)$ have trivial left Engel sinks. These examples show that the conclusion of \cite[Theorem~1.3]{khu-shu204} giving a bound for the index of $F_2(G)$ cannot be improved to a bound for the index of $F(G)$.
\end{example}

\begin{example}\label{ex-5}
Let  $p$ be an odd prime and let $A=\langle a_1\rangle \times \langle a_2\rangle \simeq \Z_p\times \Z_p$ be a direct product of two copies of $p$-adic integers regarded as procyclic pro-$p$ groups with (topological) generators $a_1,a_2$. Let $B=\langle b\rangle \simeq \Z_p$ be another  procyclic pro-$p$ group with generator  $b$. We define an action of $b$ by an automorphism on  $\langle a_1\rangle$ by setting $a^b=a^{p+1}$. Then we define the action of $b$  by an automorphism on  $\langle a_2\rangle$ as the inverse of the automorphism $a_2\mapsto a_2^{p+1}$. The resulting semidirect product $G=A\rtimes B$ admits an automorphism $\f$ of order $2$ such that $b^\f=b^{-1}$ and $a_1^\f=a_2$. Then $C_G(\f)\leq A$ and therefore  all left Engel sinks of fixed points of $\f$ are trivial.   This example shows that a pronilpotent group with a coprime automorphism of prime order all of whose fixed points have trivial left Engel sinks does not have to have an open locally nilpotent subgroup, in contrast to Theorem~\ref{t-aks} concerning right Engel sinks.
\end{example}

 \section*{Acknowledgements}
 The second author was supported by FAPDF and CNPq-Brazil.

\ed